\documentclass{amsart}
              [1997/02/02 v1.2e PROC Author Class]

\copyrightinfo{2006}
  {American Mathematical Society}

\newtheorem{theorem}{Theorem}[section]
\newtheorem{corollary}[theorem]{Corollary}

\newtheorem{proposition}[theorem]{Proposition}

\newtheorem{question}[theorem]{Question}
\newtheorem{remark}[theorem]{Remark}

\begin{document}

\title{Diagonals of separately continuous multi-valued mappings}

\author{O.G.Fotiy, V.V.Mykhaylyuk, O.V.Sobchuk}
\address{Department of Mathematics\\
Chernivtsi National University\\ str. Kotsjubyn'skogo 2,
Chernivtsi, 58012 Ukraine}
\email{ofotiy@ukr.net, vmykhaylyuk@ukr.net, ss220367@ukr.net}



\commby{Ronald A. Fintushel}



\begin{abstract}
We solve a problem on a construction of a separately continuous mapping with the given diagonal, which is the pointwise limit of a
sequence of continuous mappings valued in an equiconnected space. We construct an example of a closed-valued separately continuous
mapping $f:[0,1]^2\to \mathbb R$ with an everywhere discontinuous diagonal. The example shows that the results on the joint continuity point set
of compact-valued separately continuous mappings can not be generalized to the case of closed-valued mappings.
\end{abstract}

\maketitle
\section{Introduction}

For a mapping $f:X^2\to Y$ the function $g:X\to Y$, $g(x)=f(x,x)$, is called {\it the diagonal of the mapping $f$.}

The problem on the construction a separately continuous function (i.e. functions which are continuous with respect each variable) $f:X^2\to\mathbb R$ with the given diagonal was solved by R.~Baire \cite{Baire} in the case of  $X=\mathbb R$. R.~Baire shows that the diagonals of separately continuous functions of two real variables are exactly the first Baire functions, i.e. functions which are pointwise limits of sequences of continuous functions. This result was generalized by many mathematicians as well in the direction of Baire classification of separately continuous mappings $f:X\times Y\to Z$ as in the direction of the construction of a separately continuous mappings $f:X^2\to Z$ with the given diagonal of the first Baire class (see \cite{KMS} and  literature given there). Most general result on the construction of a separately continuous function with the given diagonal was obtained in \cite[Corollary 3.2]{KMS}. It was shown in \cite{KMS} that for every topological space $X$, metrizable equiconnected space $Z$ and function $g:X\to Z$ of the first Baire class there exists a separately continuous function $f:X^2\to Z$ with the diagonal $g$. In this connection the following question naturally arises: can we weaken the equiconnectivity of the space $Z$, in particular, can we generalize this result to the case when the diagonal $g$  is the pointwise limit of a sequence of continuous functions which valued in a equiconnected subspace $Z_1$ of the space $Z$?

On other hand, in the separately continuous mappings theory the following question naturally arises: it is possible to transfer results on separately continuous mappings on the case of multivalued mappings. For compact-valued mappings with values in a metrizable space such transfer is quite simple, because Vietoris topology on a space of nonempty compact subsets of a metrizable space is generated by Hausdorff metric \cite[c.62]{S}. The joint continuity of separately continuous closed-valued mappings $f:X\times Y\to\mathbb R$ was study in \cite{MMF}. An analog of Calbrix-Troallic theorem \cite{CT} on the set of the continuity points on horizontales was proved in \cite{MMF}. Namely, it was shown that for a topological space $X$, a $T_1$-space $Y$, a point $y_0\in Y$, which has a countable base of neighborhoods,  a metrizable locally compact $\sigma$-compact space $Z$ and a separately continuous closed-valued mapping $f:X\times Y\to Z$ there exists a residual set $A\subseteq X$  such that $f$ is jointly continuous at every point of the set $A\times\{y_0\}$. It remains non-established the following question: is it possible to transfer on the case of closed-valued mappings the Calbrix-Troallic result on the set of vertical lines which contained in the jointly continuity points set. In particular, the following question naturally arise.

\begin{question}\label{q:1} Let $f:[0,1]^2\to \mathbb R$ be a closed-valued separately continuous mapping.

$a)$\,\,Exist there a residual set $A\subseteq [0,1]$ such that $f$ is jointly continuous at every point of the set $A\times [0,1]$?

$b)$\,\,Exist there for every continuous function $\varphi:[0,1]\to[0,1]$ a residual set $A\subseteq [0,1]$ such that $f$ is jointly continuous at every point of the set $\{(x,\varphi(x)):x\in A \}$?
\end{question}

In this paper we firstly generalize a result from \cite{KMS} on the construction of separately continuous function with the given diagonal. Further, using this fact we give a negative answer on the both parts of Question \ref{q:1}. Moreover, we construct an example of a closed-valued separately continuous mapping $f:[0,1]^2\to \mathbb R$ with an everywhere discontinuous diagonal, and construct examples which point to the essentiality of some condions on obtained results.

\section{Diagonals of separately continuous mappings with values in equiconnected spaces}

Let $X$ be a topological space and $\Delta=\{(x,x)\in X^2:x\in X\}$. A set  $A\subseteq X$ is called {\it equiconnected in $X$} if there exists a continuous mapping  $\lambda:((X\times X)\cup \Delta)\times [0,1]\to X$ such that

$(i)$\,\,\, $\lambda(A\times A\times [0,1])\subseteq A)$;

$(ii)$\,\,\, $\lambda(x,y,0)=\lambda(y,x,1)=x$ for every $x,y\in A$;

$(iii)$\,\,\,$\lambda(x,x,t)=x$ for every $x\in X$ and $t\in [0,1]$.

\noindent A space is {\it equiconnected} if it is equiconnected in itself.

The following result generalize Theorem 3.1 from \cite{KMS}.

\begin{theorem}\label{th:3.1}
Let $X$ be a topological space, $Z$ be a Hausdorff space, $(Z_1,\lambda)$ be a equiconnected subspace of $Z$, $g:X\to Z$, $(G_n)_{n=0}^{\infty}$ and  $(F_n)_{n=0}^{\infty}$ be sequences of functionally open in $X^2$ sets $G_n$ and functionally closed in $X^2$ sets $F_n$ respectively, $(\varphi_n)_{n=1}^{\infty}$ be a sequence of separately continuous functions $\varphi_n:X^2\to [0,1]$, $(g_n)^{\infty}_{n=1}$ be a sequence of continuous mappings $g_n:X\to Z_1$,  which satisfy the following conditions:

$(1)$\,\,\, $G_0=F_0=X^2$ and $\Delta=\{(x,x):x\in X\}\subseteq G_{n+1}\subseteq F_n\subseteq G_n$ for every $n\in\mathbb N$;

$(2)$\,\,\, $X^2\setminus G_n\subseteq\varphi_n^{-1}(0)$ ³ $F_n\subseteq \varphi_n^{-1}(1)$ for every $n\in\mathbb N$;

$(3)$\,\,\,$\lim\limits_{n\to\infty}\lambda(g_n(x_n),g_{n+1}(x_n),t_n)=g(x)$ for every $x\in X$ and every sequence $(t_n)_{n=1}^{\infty}$ of points $t_n\in [0,1]$ and sequence $(x_n)_{n=1}^{\infty}$ of points $x_n\in X$ such that $(x_n,x)\in F_{n-1}$ for all $n\in\mathbb N$.

Then the mapping $f:X^2\to Y$,
 \begin{equation}\label{eq:3.0}
 f(x,y)=\left\{\begin{array}{ll}
                         \lambda(g_n(x),g_{n+1}(x),\varphi_n(x,y)), & (x,y)\in F_{n-1}\setminus F_n\\
                         g(x), & (x,y)\in E=\bigcap\limits_{n=1}^{\infty} G_n.
                       \end{array}
 \right.
 \end{equation}
is separately continuous.
\end{theorem}

\begin{proof} We fix $n\in\mathbb N$ and show that
\begin{equation}\label{eq:3.1}
f(x,y)=\lambda(\lambda(g_n(x),g_{n+1}(x),\varphi_n(x,y)), g_{n+2}(x),\varphi_{n+1}(x,y))
\end{equation}
for all $(x,y)\in F_{n-1}\setminus F_{n+1}$.

Let $(x,y)\in F_{n-1}\setminus F_{n}$. Since $G_{n+1}\subseteq F_n$, $(x,y)\not\in G_{n+1}$ and according to $(2)$, $\varphi_{n+1}(x,y)=0$. Therefore, $$\lambda(\lambda(g_n(x),g_{n+1}(x),\varphi_n(x,y)), g_{n+2}(x),\varphi_{n+1}(x,y))=\lambda(g_n(x),g_{n+1}(x),\varphi_n(x,y))=f(x,y).$$

Now let $(x,y)\in F_{n}\setminus F_{n+1}$. Then according to $(2)$, $\varphi_{n}(x,y)=1$ and $$\lambda(\lambda(g_n(x),g_{n+1}(x),\varphi_n(x,y)), g_{n+2}(x),\varphi_{n+1}(x,y))=$$ $$=\lambda(g_{n+1}(x),g_{n+2}(x),\varphi_{n+1}(x,y))=f(x,y).$$

Using the continuity of mappings $\lambda$, $g_n$, $g_{n+1}$, $g_{n+2}$, $\varphi_n$ and $\varphi_{n+1}$ we obtain that $f$ is continuous on the open set $G_n\setminus F_{n+1}$ for every $n\in\mathbb N$, as the composition of continuous mappings. Moreover, $f$ is continuous on the open set $G_0\setminus F_1=F_0\setminus F_1$. Therefore, $f$ is continuous on the open set $X^2\setminus E=\bigcup\limits_{n=1}^{\infty}(G_{n-1}\setminus F_n)$.

It remains to verify the continuity of the mappings $f$ with respect to variables $x$ and $y$ at all points of the set $E$. Note that $g(x)=g(y)$ for every  $(x,y)\in E$. Really, it follows from the condition $(3)$ that $\lim\limits_{n\to\infty}g_n(x)=g(x)$. Moreover, since $(x,y)\in F_{n-1}$ for every $n\in\mathbb N$, it follows from $(4)$ that $\lim\limits_{n\to\infty}g_n(x)=g(y)$. Equality $g(x)=g(y)$ follows from the fact that $Z$ is a Hausdorff space.

Now the continuity of $f$ with respect every variable at all points of the set $E$ follows from $(3)$.
\end{proof}

\begin{corollary}\label{cor:3.2}
Let $X$ be a topological space, $Z$ be a metrizable space, $(Z_1,\lambda)$ be an equiconnected in $Z$ subset of the space $Z$, and  $g:X\to Z$, besides there exists a sequence of continuous functions $g_n:X\to Z_1$ which converges to $g$ pointwise on $X$. Then there exists a separately continuous mapping $f:X^2\to Z$ with the diagonal $g$.
\end{corollary}

\begin{proof}
Fix a metric $d$ on the space $Z$ which generates its topology. Put $G_0=F_0=X^2$, $$G_n=\{(x,y)\in X^2: d(g_{k}(x),g_k(y))<\frac1n,\,\,k=1,\dots , n+1\}$$ and  $$F_n=\{(x,y)\in X^2: d(g_{k}(x),g_k(y))\leq\frac{1}{n+1},\,\,k=1,\dots ,n+2\}.$$ All sets $G_n$ are functionally open, and all sets $F_n$ are functionally closed. Therefore for every $n\in\mathbb N$ there exists a continuous function $\varphi_n:X^2\to[0,1]$ with $\varphi_n^{-1}(0)=X^2\setminus G_n$ and $\varphi_n^{-1}(1)=F_n$. It follows from the continuity of the function $\lambda:((Z_1\times Z_1)\cup \Delta)\times [0,1]\to Z$ at all points of the diagonal $\Delta=\{(z,z):z\in Z\}$ that $\lim\limits_{n\to\infty}\lambda(z_n,z_{n+1},t_n)=z$ for every sequence $(z_n)_{n=1}^{\infty}$ of points $z_n\in Z_1$ which converges to $z\in Z$ and every sequence $(t_n)_{n=1}^{\infty}$ of points $t_n\in [0,1]$.

Now using Theorem \ref{th:3.1} we obtain a separately continuous mappings $f$ with the diagonal $g$.
\end{proof}

\section{Construction of separately continuous multi-valued mappings with everywhere discontinuous diagonal}

For a topological space $Y$ by ${\mathcal F}(Y)$ we denote the space of all nonempty closed subsets of the space $Y$ with Vietoris topology. Multi-valued mapping \mbox{$f:X\to {\mathcal F}(Y)$} we denote by $f:X \to Y$.

Let $X$, $Y$ be a topological space. A multi-valued mappings $f:X\to Y$ is called {\it upper (lower) continuous at $x_0\in X$} if for every open in $Y$ set $V$ such that $f(x_0)\subseteq V$ ($f(x_0)\cap V\ne\emptyset$) there exists a neighborhood $U$ of $x_0$ in $X$ such that $f(x)\subseteq V$ ($f(x)\cap V\ne\emptyset$) for every $x\in U$. A multi-valued mapping $f$ which upper and lower continuous at $x_0$ is called {\it continuous at $x_0$}.

A family $(A_s)_{s\in S}$ of subsets $A_s$ of a topological space $X$ is called {\it discrete} if for every $x\in X$ there exists a neighborhood $U$ of $x$
such that the set $\{s\in S:A_s\cap U\ne\emptyset\}$ contains at most an one element.

\begin{theorem}\label{th:4.1}
Let $Y$ be a topological space in which there exists a discrete sequence $(Y_n)_{n=1}^{\infty}$ of subspace $Y_n$ which are homeomorphic to $[0,1]$. Then there exists a separately continuous closed-valued mapping $f:\mathbb R^2\to Y$ with everywhere discontinuous diagonal.
\end{theorem}

\begin{proof}
For every $n\in\mathbb N$ we denote by $\varphi_n$ some homeomorphism $\varphi_n:[0,1]\to Y_n$. We denote by $Z$ the space ${\mathcal F}(Y)$ and denote by $Z_1$ the subspace of $Z$ which consists of all sets $$\varphi_1([0,a_1])\cup\dots \cup \varphi_n([0,a_n])\cup\bigcup\limits _{k>n}\{\varphi_k(0)\},$$ where $n\in\mathbb N$ and $a_1,\dots , a_n\in [0,1]$. We consider the function $\lambda:Z_1^2\times [0,1]\to Z_1$ which defined by $$\lambda(u,v,t)=\varphi_1([0,(1-t)a_1+tb_1])\cup\dots \cup \varphi_n([0,(1-t)a_n+tb_n])\cup\bigcup\limits _{k>n}\{\varphi_k(0)\},$$ where $u=\varphi_1([0,a_1])\cup\dots \cup \varphi_n([0,a_n])\cup\bigcup\limits _{k>n}\{\varphi_k(0)\}, v=\varphi_1([0,b_1])\cup\dots \cup \varphi_n([0,b_n])\cup\bigcup\limits _{k>n}\{\varphi_k(0)\}\in Z_1$ and $t\in [0,1]$. It easy to se that the mapping $\lambda$ is continuous, i.e. the space $(Z_1,\lambda)$ is equiconnected.

For a convenience of notation we consider the interval $X=(0,1)$ instead the real line $\mathbb R$. Consider closed-valued mapping $g: (0,1)\to Y$ which defined by  $$g(x)=\bigcup\limits_{n=1}^{\infty}\varphi_n([0,x]).$$ Show that the mapping $g$ is everywhere discontinuous, namely, $g$ is not upper continuous at every point $x\in (0,1)$. Fix a point $x_0\in (0,1)$. The set $U=\bigcup\limits_{n=1}^{\infty}\varphi_n([0,x_0+\frac{1}{n})\cap[0,1])$ is open in the space $g((0,1))$ and $g(x_0)\subseteq U$. But the set $\{x\in (0,1):g(x)\subseteq U\}=(0,x_0]$ is not a neighborhood of $x_0$, i.e. $g$ is not upper continuous at $x_0$.

Put $G_0=F_0=X^2$, $G_n=\{(x,y)\in X^2: |x-y|<\frac{1}{n+2}\}$ and $F_n=\{(x,y)\in X^2: |x-y|\leq\frac{1}{n+3}\}$ for every $n\in\mathbb N$. Moreover, for every $n\in\mathbb N$ we consider the continuous closed-valued mapping $g_n:X\to Y$,
$$g_n(x)=\left\{\begin{array}{ll}
                         \bigcup\limits_{n=1}^{\infty}\{\varphi_n(0)\}, & x\in (0,\frac{1}{n+1}],\\
                         \bigcup\limits_{k=1}^{n}\varphi_k([0,x-\frac{1}{n+1}])\cup\bigcup\limits _{k>n}\{\varphi_k(0)\}, & x\in (\frac{1}{n+1},1).
                       \end{array}
 \right.$$ Note that $g_n(X)\subseteq Z_1$ for every $n\in\mathbb N$. Since  $g_n(x)\subseteq g(y)$ for every  $n\in\mathbb N$ and $x,y\in X$ ç $(x,y)\in F_{n-1}$, the sequence $(g_n)_{n=1}^{\infty}$ satisfies the condition $(3)$ of Theorem \ref{th:3.1}. It remains to use Theorem \ref{th:3.1}.
\end{proof}

\begin{remark}\label{r:4.1} For the case $X=[0,1]$ similar theorem can be proved analogously. Herewith it is sufficient to consider the mapping $g: X\to Y$,
$$g(x)=\left\{\begin{array}{ll}
                         \bigcup\limits_{n=1}^{\infty}\varphi_n([0,x]), & x\in [0,1),\\
                         \bigcup\limits_{n=1}^{\infty}\{\varphi_n(0)\}, & x=1,
                       \end{array}
 \right.$$ and corresponding sequence of continuous mappings $g_n:X\to Y$,
$$g_n(x)=\left\{\begin{array}{lll}
                         \bigcup\limits_{n=1}^{\infty}\{\varphi_n(0)\}, & x\in (0,\frac{1}{n+1}]\cup [\frac{n}{n+1},1],\\
                         \bigcup\limits_{k=1}^{n}\varphi_k([0,x-\frac{1}{n+1}])\cup\bigcup\limits _{k>n}\{\varphi_k(0)\}, & x\in (\frac{1}{n+1},\frac{n^2+1}{(n+1)^2}),\\
                         \bigcup\limits_{k=1}^{n}\varphi_k([0,-nx+\frac{n^2}{n+1}])\cup\bigcup\limits _{k>n}\{\varphi_k(0)\}, & x\in [\frac{n^2+1}{(n+1)^2},\frac{n}{n+1}).
                       \end{array}
 \right.$$
\end{remark}

The following corollary gives the negative answer to Question \ref{q:1}.

\begin{corollary}\label{cor:4.3}
There exists a closed-valued separately continuous mapping $f:[0,1]^2\to \mathbb R$ with the everywhere discontinuous diagonal.
\end{corollary}

\section{Examples}

The following example indicates that in Corollary \ref{cor:3.2} the condition of the equiconnectivity of $(Z_1,\lambda)$ in $Z$ can not be weakened to the condition of the equiconnectivity of $(Z_1,\lambda)$.

\begin{proposition}\label{pr:5.1}
 There exist a metrizable space $Z$, an equiconnected subspace $Z_1$ of space $Z$  and a function $g:[0,1]\to Z$ of the first Baire class such that the following condions hold:

 $(1)$\,\,there exists a sequence of continuous function $g_n:[0,1]\to Z_1$ which converges to $g$ pointwise on $[0,1]$;

 $(2)$\,\,the function $g$ is not the diagonal of any separately continuous function $f:[0,1]^2\to Z$.
\end{proposition}

\begin{proof}
We show that the spaces $$Z=\{(0,1), (0,-1)\}\cup\bigcup\limits_{n=1}^{\infty}\{(x,nx):x\in\mathbb R\},$$ $$Z_1=\bigcup\limits_{n=1}^{\infty}\{(x,nx):x\in\mathbb R\}$$ with the topologies induced by $\mathbb R^2$ and the function
$$g(x)=\left\{\begin{array}{ll}
                         (0,1), & x\in [0,\frac{1}{2}),\\
                         (0,-1), & x\in[\frac{1}{2},1],
                       \end{array}
 \right.$$ are desired. It easy to see that $Z_1$ is equiconnected. For the proof of $(1)$ it is sufficient to consider a sequence of the functions

$$g_n(x)=\left\{\begin{array}{lll}
                         (\frac{1}{n},1), & x\in [0,\frac{n-1}{2n}],\\
                         (-4x+\frac{2n-1}{n},-4nx+2n-1), & x\in (\frac{n-1}{2n},\frac{1}{2}],\\
                         (-\frac{1}{n},-1), & x\in (\frac{1}{2},1].
                       \end{array}
 \right.$$ It remains to verify $(2)$.

Suppose that there exists a separately continuous function $f:[0,1]^2\to Z$ with the diagonal $g$. Since for every $n\in\mathbb N$ the set $\{(x,nx):x\in\mathbb R\}\setminus \{(0,0)\}$ is an open-closed set in $Z\setminus\{(0,0)\}$, the sets $\{(0,1)\}$ and $\{(0,-1)\}$ are components of the connectivity in the space $Z\setminus\{(0,0)\}$. This implies that every connected set which contains at least one of points $(0,1)$ and $(0,-1)$ either is one-pointed or contains the point $(0,0)$. Therefore the function $f$ is vertically and horizontally constant. Thus, $f$ is a constant which gives a contradiction.
\end{proof}

In the connection with Theorem \ref{th:4.1} and Corollary \ref{cor:4.3} the following question naturally arises: exist there for every mapping $g:[0,1]\to {\mathcal F}(\mathbb R)$ of the first Baire class a separately continuous mapping $f:[0,1]^2\to \mathbb R$ with the diagonal $g$?
The following example gives a negative answer to this question.

\begin{proposition}\label{pr:5.2}
There exists a closed-valued mapping $g:[0,1]\to \mathbb R$ of the first Baire class such that $g$ is not the diagonal for any separately continuous closed-valued mapping $f:[0,1]^2\to \mathbb R$.
\end{proposition}

\begin{proof}
We consider the closed-valued mapping $g:[0,1]\to \mathbb R$, $g(x)=\bigcup\limits_{n=1}^{\infty}\{x+n\}$. Note that for every $n\in\mathbb N$ the closed-valued mapping $g_n:[0,1]\to \mathbb R$, $g_n(x)=\bigcup\limits_{k=1}^{n}\{x+n\}$, is continuous. Moreover, $g_n\to g$ pointwise on $X$. Therefore the mapping $g$ is of the first Baire class.

Suppose that there exists a separately continuous closed-valued mapping $f:[0,1]^2\to \mathbb R$ such that $f(x,x)=g(x)$ for every $x\in X$. According to Proposition 3.3 from \cite{MMF} for every $x\in[0,1]$ there exists an integer $n_x\in\mathbb N$ such that $$\left(f(x,y)\cup f(y,x)\right)\cap\left((-\infty,-n_x)\cup (n_x,+\infty)\right)\subseteq g(x)$$ for every $y\in [0,1]$ with $|x-y|<\frac{1}{n_x}$. We choose $n\in\mathbb N$, an open nonempty set $U\subseteq [0,1]$ and dense in $U$ set $A$ such that $n_x\leq n$  for every $x\in A$. Without loss of the generality we can propose that ${\rm diam}(U)<\frac{1}{n}$. Then $$f(x,y)\cap\left((-\infty,-n_x)\cup (n_x,+\infty)\right)\subseteq g(x)\cap g(y)$$ for every $x,y\in A$. Since $g(x)\cap g(y)=\emptyset$ for every distinct $x,y\in[0,1]$, $f(x,y)\subseteq [-n,n]$ for every $x,y\in A$. Now the separate continuity of $f$ and density of $A$ in $U$ imply that $f(x,y)\subseteq[-n,n]$ for every $x,y\in U$. But this contradicts to the fact that $g$ is the diagonal of $f$.
\end{proof}

\bibliographystyle{amsplain}

\end{document}